\documentclass{cc}

\usepackage{graphicx}

\usepackage{lmodern}

\usepackage{amsfonts}
\usepackage{amsmath}
\usepackage{mathtools}

\usepackage{algorithm}
\usepackage{algorithmic}
\usepackage{enumitem}

\DeclarePairedDelimiter\abs{\lvert}{\rvert}%

\newcommand{\ux}{\underline{x}}
\newcommand{\ox}{\overline{x}}

\newcommand{\bx}{\boldsymbol{x}}
\newcommand{\QQ}{\mathbb{Q}}
\newcommand{\RR}{\mathbb{R}}
\newcommand{\EE}{\mathsf{E}}
\newcommand{\T}{^\mathrm{T}}

\newcommand{\fracn}{\frac{1}{n}}

\colorlet{MAGENTA}{magenta}

\renewcommand{\ln}{\log}

\date{May 10, 2022}

\contact{ondrej.sokol@vse.cz}
\received{May 10, 2022}
\submitted{May 10, 2022}
           \title{The NP-hard problem of computing the maximal sample variance over interval data is solvable in almost linear time with high probability}
\titlehead{The maximal variance over interval data}
\author{
           M. Rada \\
           Prague University of Economics and Business \\
           W.~Churchill Square~4, Prague~3 \\ Czech Republic \\
           \email{miroslav.rada@vse.cz}
          \and   
M. \v Cern\'{y} \\
               Prague University of Economics and Business \\
               W.~Churchill Square~4, Prague~3 \\ Czech Republic   \\
               \email{cernym@vse.cz}
           \and
            O. Sokol \\
               Prague University of Economics and Business \\
               W.~Churchill Square~4, Prague~3 \\ Czech Republic   \\
               \email{ ondrej.sokol@vse.cz} 
           }

\begin{abstract}
        We consider the algorithm by Ferson et al.~(\emph{Reliable computing} {\bf 11}(3), p.~207--233, 2005) designed for solving the NP-hard problem of computing the maximal sample variance over interval data, motivated by robust statistics (in fact, the formulation can be written as a 
        nonconvex quadratic program with a specific structure).
        First, we propose a new version of the algorithm improving its original time bound $O(n^2 2^\omega)$ to $O(n \log n+n\cdot 2^\omega)$,
        where $n$ is number of input data and $\omega$ is the clique number in a certain intersection graph. Then we treat input data as random variables (as it is usual in statistics) 
and introduce a natural probabilistic data generating model. We get $2^\omega = O(n^{1/\log\log n})$ 
and $\omega = O(\log n / \log\log n)$ on average. This results in average computing time $O(n^{1+\epsilon})$
for $\epsilon > 0$ arbitrarily small, which may be considered as ``surprisingly good'' average time complexity for solving an NP-hard problem. 
Moreover, we prove the following tail bound on the distribution of computation time: hard instances, forcing
the algorithm to compute in time $2^{\Omega(n)}$, occur rarely,
with probability tending to zero at the rate $e^{-n\log\log n}$.
\end{abstract}


\begin{subject}
Quadratic programming, Geometric probability and stochastic geometry 
\end{subject}

\begin{document}

\section{Introduction and motivation}

\subsection{Problem formulation}
\cite{ferson:2005:ExactBoundsFinitea} studied the pair of optimization problems
\begin{align}
 \min_{x \in \RR^n} V(x)\quad\text{s.t.}\quad\ux \leq x \leq \ox, \label{eq:lbound}\\
\max_{x \in \RR^n} V(x)\quad\text{s.t.}\quad\ux \leq x \leq \ox, \label{eq:ubound}
\end{align}
where
and $\ux \leq \ox \in \QQ^n$ are given input data and
$$
V(x) \coloneqq \frac{1}{n}\sum_{i=1}^n \left(x_i - \frac{1}{n}\sum_{j=1}^n x_j\right)^2.
$$
It is obvious that \whole\ref{eq:lbound} is a convex quadratic program (CQP) solvable in polynomial time, while \whole\ref{eq:ubound} is easily proven to be NP-hard. It is worth noting that a general CQP solver yields a weakly polynomial algorithm for \whole\ref{eq:lbound}, but \cite{ferson:2005:ExactBoundsFinitea} introduced a strongly polynomial method. 

They also introduced a method for solving~\whole\ref{eq:ubound} which works in exponential time in the worst case (not surprisingly). 
The method will be described in \whole\ref{sect:algo}. Abbreviating the names of all authors (Ferson, Ginzburg, Kreinovich, Longpré and Aviles), we will refer to their method as \emph{FGKLA algorithm}. 

\subsection{Summary of results}
In this text we focus on the NP-hard case \whole\ref{eq:ubound} and the FGKLA algorithm.
Our contribution is twofold.
\paragraph{Improving the worst-case complexity of the FGKLA algorithm.} Firstly, we show that there exists an implementation of the FGKLA algorithm working in time 
    \begin{equation}
    O(n\log n + n \cdot 2^{\omega}), \label{eq:impbnd}    
    \end{equation}
    where $\omega$ is the 
    size of the largest clique in a certain intersection graph. The graph will be introduced in \whole\ref{def:intgraph}. This improves the bound $O(n^2 \cdot 2^{\omega})$ from the original paper. For further discussion see \whole\ref{rem:orig:impl:problem}.
    \paragraph{Proving a ``good'' behavior in a probabilistic setting.} Secondly, we treat the input data $\ux, \ox$ as random variables. We introduce a natural and fairly general probabilistic model (details are in \whole\ref{sect:probability}), under which we show 
    that 
    \begin{enumerate}[label=(\roman*)]
        \item \label{enu:props:i} on average, the algorithm works in time 
        \begin{equation}
          O(n^{1+\epsilon})\quad \text{for all}\ \epsilon > 0 \label{eq:timebnd},
        \end{equation}
        which is surprisingly good considering the problem is NP-hard,
        \item \label{enu:props:ii} the probability that the algorithm computes in time $2^{\Omega(n)}$ 
tends to zero \emph{faster than exponentially} with $n \rightarrow \infty$. In other words, we show that ``hard'' instances occur indeed rarely.
    \end{enumerate}

    More specifically: \whole\ref{enu:props:i}~we prove that under the probabilistic model it holds 
\begin{equation}
\EE2^\omega = O(n^\frac{1}{\log\log n}), 
\label{eq:avgomega}
\end{equation}
where $\EE[\cdot]$ stands for the expected value of $[\cdot]$. Combination of \whole\ref{eq:avgomega} with \whole\ref{eq:impbnd} yields \whole\ref{eq:timebnd} as 
$n\log n = O(n^{1+\epsilon})$ and
$n^\frac{1}{\log \log n} = O(n^{\epsilon})$ for any $\epsilon > 0$. In the entire text, 
``$\log$'' stands for the natural logarithm.

To achieve~\whole\ref{enu:props:ii}: from \whole\ref{eq:impbnd} it follows that the computing time is exponential when $\omega \geq \delta n$ with $\delta > 0$. We prove
that
$$
\Pr[\omega \geq \delta n] \leq e^{-n\ln \ln n}  \text{ for every}\ \delta > 0 \text{\ and a sufficiently large $n$}. 
$$

\subsection{Motivation from statistics}\label{sect:statistics}
Problems \whole\ref{eq:lbound} and \whole\ref{eq:ubound} are studied in statistics; see e.g. \cite{antoch:2010:notevariabilityinterval}
and references therein,
and a pseudopolynomial method in \cite{cerny:2014:complexitycomputationapproximationa}.
The statistical motivation is as follows: we are interested in sample variance $V(x)$ of a dataset $x = (x_1, \dots, x_n)\T$.
However, the data $x$ is not observable. What is available instead is a collection of intervals $\bx_i \coloneqq[\ux_i, \ox_i]$, $i = 1, \dots, n$,
such that $\ux_i \leq x_i \leq \ox_i$ (for example, instead of the exact values $x$ we have rounded versions only). Then, $V(x)$
cannot be computed exactly, but we can get tight bounds for $V(x)$ in the form \whole\ref{eq:lbound} and \whole\ref{eq:ubound}. 
In econometrics, this phenomenon is sometimes called \emph{partial identification}, see  \cite{manski:2003:Partialidentificationprobability}.

The problem is more general and is studied for various statistics in place of $V(x)$ in \whole\ref{eq:lbound} and \whole\ref{eq:ubound},
see the reference books 
by  \cite{kreinovich:1998:ComputationalComplexityFeasibility} and  \cite{nguyen:2012:ComputingStatisticsInterval}.

\subsection{Related work} In general, this paper contributes to the analysis of complexity of optimization problems and algorithms when input data can be assumed to be random, drawn from a particular distribution or a class of distributions. As a prominent example recall the famous average-time analysis of the Simplex Algorithm by \cite{borgwardt:1982:Averagenumberpivot} and \cite{spielman:2004:Smoothedanalysisalgorithms}, where the phenomenon ``exponential in the worst case but fast on average'' has been studied since 1980's.

The phenomenon is particularly interesting for NP-hard problems since the worst-case exponential time seems to be unavoidable. In the area of quadratic optimization, the simplex-constrained case has been studied by several authors (see e.g.  \cite{bomze:2017:ComplexitySimpleModels} and references therein). It turns out that a quadratic form with entries randomly generated from ``natural'' distributions attains, with a high probability, its global maximum in a face of a small dimension. This property implies that the problem can be solved efficiently by enumerating faces. 

Another nice example is the analysis of the average-case complexity of an NP-hard variant of the open shop scheduling problem by \cite{lu:1993:NPHardOpenShop}. Their setup is similar to ours: they assume that input data (the job processing times) are generated from a certain class of probabilistic distributions and prove that the average complexity is polynomial in the number of jobs.

Finally, we mention the average-case complexity analysis of the NP-hard \textsc{$k$-clique} problem.  \cite{rossman:2014:MonotoneComplexitykClique} derived bounds on average-case complexity on monotone circuits. \cite{fountoulakis:2015:averagecasecomplexityparameterized} then extended the analysis into a probabilistic setting, showing whether the ``hard'' instances occur frequently or rarely. Their results are, in a sense, analogous to ours: if the edges are sampled from a ``natural'' distribution, then \textsc{$k$-clique} can be solved in polynomial time with probability tending to one faster than polynomially.

Aside from the above mentioned general average-case results, the \emph{weak average-case complexity} paradigm by \cite{ame} is worth mentioning; here, the word ``weak'' refers to omitting a ``small'' subset of hard instances from the set of all possible instances and performing the average-case complexity analysis on such a reduced class of instances. It turns out that for many interesting problems the removal of a subset with exponentially small measure is sufficient to making hard problems tractable.

\section{FGKLA Algorithm} \label{sect:algo}
Recall that the input instance is given by the pair $\ux = (\ux_1, \dots, \ux_n)\T$
and $\ox = (\ox_1, \dots, \ox_n)\T$. 
Compact intervals will be denoted in boldface, e.g.~$\bx_i = [\ux_i, \ox_i]$. 
For $i = 1, \dots, n$ define
$$
x^*_i \coloneqq \tfrac{1}{2}(\ux_i + \ox_i),\quad x^{\Delta}_i \coloneqq \tfrac{1}{2}(\ox_i - \ux_i),
\quad \bx_i^{1/n} \coloneqq [x^*_i - \tfrac{1}{n}x_i^{\Delta}, x^*_i + \tfrac{1}{n}x_i^{\Delta}].
$$
The numbers $x_i^*, x_i^{\Delta}$ are referred to as \emph{center} and \emph{radius} of $\bx_i$, respectively, and $\bx_i^{1/n}$ is called
a \emph{narrowed interval} (i.e., $\bx_i$ shrunk by factor $n$ around its center). 
For $x \in \RR^n$ we define $\mu[x] \coloneqq \frac{1}{n}\sum_{i=1}^n x_i$
(the \emph{mean} of $x$).

Our version of the FGKLA algorithm is summarized as Algorithm \whole\ref{alg:fgkla}.
The main result of this section is \whole\ref{lem:fcplx}. In particular, it improves the 
worst-case complexity bound $O(n^2\cdot 2^\omega)$
from~\cite{ferson:2005:ExactBoundsFinitea} (see also \whole\ref{rem:orig:impl:problem}). 
The proof of \whole\ref{lem:fcplx} will be given in \whole\ref{ssect:idea}.

\begin{theorem}[properties of the FGKLA algorithm (Algorithm \whole\ref{alg:fgkla})]
\begin{itemize}
    \item[(a)] The FGKLA algorithm correctly solves \whole\ref{eq:ubound}.
    \item[(b)] Let $G = (V = \{1, \dots, n\}, E)$ be an undirected graph where $\{i, j\} \in E$ if and only if $\bx^{1/n}_i \cap \bx^{1/n}_j\neq \emptyset$ (here, $i \neq j$). Let $\omega$ be the size of the largest clique in $G$. Then, FGKLA algorithm works in time $O(n \log n +  n \cdot 2^{\omega})$.
\label{lem:fcplx}
\end{itemize}
\end{theorem}

\begin{definition}\label{def:intgraph}
The graph $G$ from \whole\ref{lem:fcplx} is referred to as \emph{FGKLA intersection graph with data $\bx_1, \dots, \bx_n$}.
\end{definition}

\subsection{Idea of the FGKLA algorithm}\label{ssect:idea}
Since the quadratic form $V(x)$ is positive semidefinite, the 
maximum of \whole\ref{eq:ubound} is attained in a vertex (an extremal point) 
of the feasible set 
$$
\bx = \{x\mid \underline{x} \le x \le \overline{x}\}.
$$ 
There are $2^n$ vertices in total. In a vertex $x$
we have 
$x_i \in \{\ux_i = x^*_i-x^\Delta_i,\ \ox_i = x_i^*+x_i^\Delta\}$ for every $i$. 
FGKLA algorithm reduces the number of vertices to be examined from $2^n$ to 
$O(n\cdot 2^\omega)$. The reduction is based on \whole\ref{lem:fixing}. A similar lemma was used in the original paper \cite{ferson:2005:ExactBoundsFinitea}.

\begin{lemma}\label{lem:fixing}
    Let $x, x' \in \bx$ and let there exist $i \in \{1,\ldots,n\}$ such that \begin{enumerate}
\item $x_j = x'_j$ for all $j \neq i$ and 
\item\label{enu:ass:2} one of the following is satisfied:
    \begin{enumerate}
        \item\label{enu:ass:a} $x_i = \ox_i$, $\mu[x] > x_i^* + \frac{1}{n}x_i^\Delta$ and $x'_i = \ux_i$,
        \item\label{enu:ass:b} $x_i = \ux_i$, $\mu[x] < x_i^* - \frac{1}{n}x_i^\Delta$ and $x'_i = \ox_i$.
    \end{enumerate}
    \end{enumerate}
    Then $V(x) < V(x')$.
\end{lemma}

\begin{proof}
We prove the claim for Assumption \ref{enu:ass:a}, i.e. $x_i = \ox_i, \mu[x] > x_i^*+\frac{1}{n}x_i^\Delta, x'_i = \ux_i$. The proof is analogous for \ref{enu:ass:b}.

    Let $J = \{1,\ldots,n\} \setminus \{i\}$.
    We want to prove $V(x) - V(x') < 0$. We have
    \begin{align*}
        &V(x) - V(x')\\
        &\quad= \fracn \Bigg(\ox_i^2 + \sum_{j \in J} x_j^2 - \ux_i^2 - \sum_{j \in J}x_j^2 \\&\qquad \qquad - \fracn\bigg( \Big(\ox_i + \sum_{j \in J} x_j \Big)^2 - \Big( \ux_i + \sum_{j \in J} x_j   \Big)^2  \bigg) \Bigg) \\
        &\quad= \fracn \bigg( \ox_i^2-\ux_i^2 - \fracn \Big( \ox_i^2 - \ux_i^2 + 2(\ox_i-\ux_i)\sum_{j\in J} x_j \Big) \bigg)\\
        &\quad= \fracn \bigg( 4x^*_ix_i^\Delta - \fracn(4 x^*_ix^\Delta_i) - \Big(4x^\Delta_i \sum_{j\in J} x_j \Big)\bigg)\\
        &\quad= \frac{4}{n} x^\Delta_i \bigg( x^*_i - \fracn x^*_i - \fracn\Big(-\ox_i + \sum_{j\in\{1,\ldots,n\}} x_j\Big) \bigg)\\
        &\quad <  \frac{4}{n} x^\Delta_i\left( x^*_i -\fracn x^*_i+\fracn x^*_i+\fracn x^\Delta_i-x^*_i-\fracn x^\Delta_i \right)
        = 0. {\qed}
    \end{align*}
\end{proof}

\begin{corollary}
    \label{cor:fixing}
    Let $x' \in \bx$ be a maximizer and $\mu' = \mu[x']$.
    Let $X$ be the set of all vectors $x \in \bx$ satisfying:
\begin{itemize}
\item[(a)] 
    $x_i = \ux_i$ if $\mu' > x_i^*+\frac{1}{n} x_i^\Delta$,
\item[(b)] 
    $x_i = \ox_i$ if $\mu' < x_i^*-\frac{1}{n} x_i^\Delta$, and
\item[(c)] 
    $x_i \in \{\ux_i, \ox_i\}$ if $\mu' \in \bx_i^{1/n}$.
\end{itemize}
Then $X$ contains a maximizer.
\end{corollary}

In cases (a) and (b) we say that variable $x_i$ (or index~$i$) is \emph{determinable} with respect to~$\mu'$;
in case (c), variable $x_i$ (or index~$i$) is \emph{free} with respect to $\mu'$.

\begin{algorithm}[t]
    \begin{algorithmic}[1]
        \normalsize
        \REQUIRE{$\ux = x^* - x^{\Delta} \in \mathbb{Q}^n,\ \ox = x^* + x^{\Delta} \in \mathbb{Q}^n$ s.t. $\ux \leq \ox$}
        \STATE \label{alg:a:construction}$A \coloneqq \{x_i^* + \frac{s_i}{n} x_i^\Delta \in [\mu[\ux],\mu[\ox]] \mid s_i\in\{\pm 1\},\ i=1,\dots,n \}$; $m \coloneqq \abs{A}$
        \STATE \label{alg:sorting}sort $A$ and denote its elements by $a_1 < \cdots < a_{m}$
        \STATE \textbf{for} $k\in \{1,\dots,m\}$ \textbf{do} $B_{a_k} \coloneqq \emptyset$; $E_{a_k} \coloneqq \emptyset$
        \STATE \textbf{for} $i\in \{1,\dots,n\}$ \textbf{do} add $i$ to both $B_{x^*_i-\frac{1}{n}x_i^\Delta}$ and $E_{x^*_i+\frac{1}{n}x_i^\Delta}$ 
        \STATE \label{alg:V:initial}$V_1 \coloneqq \mu[(\ox_1^2,\ldots,\ox_n^2)]$; $V_2 \coloneqq \mu[\ox]$; $M \coloneqq V_1 - V_2^2$; $L \coloneqq \emptyset$
        \FOR{$k \in \{1,\dots,m\}$}\label{alg:f:for:start}
        \STATE \label{alg:B:processing}\textbf{for} {$i \in B_{a_k}$} \textbf{do} $L \coloneqq L \cup \{i\}$
        \STATE\label{alg:L:processing}examine all $2^{\abs{L}}$ vertices with Algorithm \whole\ref{alg:traversal}
        \STATE \label{alg:E:processing}\textbf{for} {$i \in E_{a_k}$} \textbf{do} $L \coloneqq L \setminus \{i\}$; $V_1 \coloneqq V_1 + \frac{1}{n}(\ux_i)^2 - \frac{1}{n}(\ox_i)^2 $; $V_2 \coloneqq V_2 - \frac{2}{n} x_i^\Delta$
        \ENDFOR\label{alg:f:for:end}
        \RETURN $M$
    \end{algorithmic}
    \caption{FGKLA algorithm}
    \label{alg:fgkla}
\end{algorithm}

\begin{algorithm}[t]
    \begin{algorithmic}[1]
        \normalsize
        \setcounter{ALC@unique}{0}
        \REQUIRE{list $L$ of free indices (variables $V_1, V_2$ are global)}
        \STATE $z\coloneqq(0,\ldots,0)\in \{0,1\}^{\abs{L}}$; $s\coloneqq(1,\ldots,1) \in \{\pm 1\}^{\abs{L}}$; $c \coloneqq 0$
        \WHILE{$c < 2^{\abs{L}}$}\label{alg:while:start}
        \FOR{$i \in \{1,\ldots,\abs{L}\}$}\label{alg:for:start}
        \STATE \textbf{if} $z_i = 0$ \textbf{then goto} Line \whole\ref{alg:set:zi}
        \STATE $z_i \coloneqq 0$
        \ENDFOR\label{alg:for:end}
        \STATE\label{alg:set:zi} $z_i \coloneqq 1$; $s_i \coloneqq -s_i$ \qquad{\footnotesize ($i$ is the value with which \textbf{for} cycle \whole\ref{alg:for:start}--\whole\ref{alg:for:end} ends)}
        \STATE\label{alg:update} $V_1 \coloneqq V_1 + \frac{1}{n}(x_{L_i}^*+s_i x_{L_i}^\Delta)^2 - \frac{1}{n}(x_{L_i}^*-s_i x_{L_i}^\Delta)^2 $; $V_2 \coloneqq V_2 + \frac{2}{n}s_i x_{L_i}^\Delta$
        \STATE\label{alg:test}\textbf{if} $V_1-V_2^2 > M$ \textbf{then} $M \coloneqq V_1 - V_2^2$
        \STATE{$c \coloneqq c+1$}
        \ENDWHILE\label{alg:while:end}
    \end{algorithmic}
    \caption{Examining vertices corresponding to free indices in~$L$}
    \label{alg:traversal}
\end{algorithm}

Algorithm \whole\ref{alg:fgkla} works as follows. It builds the set $A$ (Line \whole\ref{alg:a:construction}) 
containing all endpoints $x_i^{*} \pm \frac{1}{n}x_i^{\Delta}$ of the narrowed intervals $\bx_i^{1/n}$, $i = 1, \dots, n$ (here, $A$ acts as a set rather than a list, meaning that possible duplicities are removed), sorts them and denotes them by $a_1<\cdots<a_m$ (Line \whole\ref{alg:sorting}).

Consider the set $\{\mu[x] \mid x \in \bx\} = [\mu[\ux],\mu[\ox]]$ of all possible means. 
The endpoints from $A$ divide the interval $[\mu[\ux],\mu[\ox]]$ into at most $2n+1$ regions 
$$[a_0\coloneqq\mu[\ux],a_1),(a_1,a_2),\ldots,(a_{m-1},a_m),(a_m,a_{m+1}\coloneqq\mu[\ox]]. 
$$
Thanks to  \whole\ref{lem:fixing} and \whole\ref{cor:fixing}, every region $(a_k, a_{k+1})$
contains means $\mu$ with the same set of free indices.
For a region $(a_k, a_{k+1})$, we denote this set by $I(a_k,a_{k+1})$, 
i.e.~$I(a_k,a_{k+1}) \coloneqq \{i \in \{1, \dots, n\}\mid \bx_i^{1/n} \cap (a_k, a_{k+1}) \neq \emptyset\}$. The set $A$ of endpoints contains \emph{the worst possible} mean values with respect to the number of free indices. More precisely: for every $a_k$, $k = 1, \dots, m$, 
all indices from $I(a_{k-1},a_k) \cup I(a_k,a_{k+1})$ are free. 

  On Line \whole\ref{alg:V:initial}, we examine the vertex $\ox$. The value of $V(\ox)$ is computed and stored as $M$, the maximal value of $V$ found so far. Variables $V_1$ and $V_2$ will be useful in Line  \whole\ref{alg:traversal}.

  Then, Algorithm  \whole\ref{alg:fgkla} takes means $a_1, \ldots, a_m$ one by one. For every mean, say $a_k$, it takes the set $B_{a_k} = \{i\mid x_i^{*}-\frac{1}{n}x_i^{\Delta} = a_k\}$ of indices of narrowed intervals beginning in $a_k$ and inserts it to the set $L$ of free indices with respect to $a_k$ (Line  \whole\ref{alg:B:processing}). Indices $\{1, \dots, n\} \setminus L$ are determinable with respect to $a_k$. This yields $2^{\abs{L}}$ candidate vertices that are examined by Algorithm \whole\ref{alg:traversal}, called on Line  \whole\ref{alg:L:processing}.

Then, indices from the set $E_{a_k} = \{i\mid x_i^{*}+\frac{1}{n}x_i^{\Delta} = a_k\}$ of narrowed intervals ending in $a_k$ are removed from $L$. Intervals with these indices will be fixed to the lower endpoint for every upcoming $k'>k$ 
(Line  \whole\ref{alg:E:processing} of~Algorithm \whole\ref{alg:fgkla}). The update of $V_1$ and $V_2$ will be explained later. 

Algorithm \whole\ref{alg:traversal} consecutively traverses all $2^{\abs{L}}$ vertices of $\bx$ resulting from fixing either $x_i = \ux_i$ or $x_i = \ox_i$ for the free indices $i \in L$. 
For every such vertex, say $x$, the variance $V(x)$ is computed. To make these computations cheap, the traversal of $L$ 
is performed in a way that \emph{two successive vertices $x, x'$ differ in just one component}. 
Then  \whole\ref{pro:variance:updates} shows how to get $V(x')$ from $V(x)$ with $O(1)$ arithmetic operations. The variance is stored indirectly as variables $V_1$ and $V_2$; they can be easily updated when $\ux_i$ is switched to $\ox_i$, or vice versa. 
\begin{lemma}
    \label{pro:variance:updates}
For $x \in \RR^n$, we have $V(x) = V_1 - V_2^2$, where 
$V_1 = \mu[(x_1^2,\ldots,x_n^2)]$ and $V_2 = \mu[x]$. 
Furthermore, if $x'$ differs from $x$ in just one component, say $i$th, then 
$$\textstyle V(x') = V_1 + \frac{1}{n}((x')^2_i - x^2_i)- (V_2 + \frac{1}{n}(x'_i - x_i))^2.$$
\end{lemma}

Algorithm \whole\ref{alg:traversal} is an adaptation of the algorithm from \citet[p. 37]{rohn:2006:Solvabilitysystemsinterval} for enumeration of elements of the set $\{\pm 1\}^\ell$ for a given $\ell$. The enumeration can in general start from an arbitrary element. The proof of correctness can be found therein. In our variant, the variable $s \in \{\pm 1\}^{\abs{L}}$ indicates the current vertex. 
In every iteration of \textbf{while} cycle, some $s_i$ is set to $-s_i$. The $i$th index $L_i$ is taken from $L$ (here we consider $L$ as a list rather than a set) and $x_{L_i}$ is switched to the other endpoint. 
For this new vertex, $V_1$ and $V_2$ are updated (Line \whole\ref{alg:update}) and the resulting variance $V$ is compared to the best value found so far (Line \whole\ref{alg:test}).

The following property of Algorithm~\whole\ref{alg:traversal} is crucial for the correctness and complexity of FGKLA algorithm: 
\emph{When Algorithm \whole\ref{alg:traversal} ends, then $s = (1, \dots, 1)$.} The next proposition immediately follows.

\begin{lemma}
Let $V_1^*, V_2^*$ be the values of the global variables $V_1, V_2$ when Algorithm \whole\ref{alg:traversal} starts
and let $V_1^{**}, V_2^{**}$ be the values of $V_1, V_2$ when Algorithm \whole\ref{alg:traversal} ends.
Then $V_1^* = V_1^{**}$ and $V_2^* = V_2^{**}$.
\end{lemma}
In particular, this means that when entering Line \whole\ref{alg:L:processing} of Algorithm \whole\ref{alg:fgkla}, we always start examining the free indices with $x_i \in \ox_i$ for each $i \in L$.

Finally, Line \whole\ref{alg:E:processing} of Algorithm \whole\ref{alg:fgkla} removes intervals ending in $a_k$ from $L$. These intervals are going to be fixed to their lower endpoints in the following iterations. 
Since they are at the upper endpoint now, Line \whole\ref{alg:E:processing} updates $V_1$ and $V_2$ accordingly.

\begin{proof}{Proof of \whole\ref{lem:fcplx}}\hfill
    \begin{enumerate}[label=\alph*)]
        \item \emph{Correctness.} Let $x\in\RR^n$ be a maximizer of \whole\ref{eq:ubound}. Since the maximum
is attained in a vertex of the feasible set $\bx$, we can assume $x_i \in \{\ux_i, \ox_i\}$ for
all~$i$. Moreover, thanks to \whole\ref{cor:fixing}, we can assume $x_i = \ox_i$ for every $i$ such that $\mu[x] < x_i^*-\frac{1}{n}x_i^\Delta$ and $x_i = \ux_i$ for every $i$ such that $\mu[x] > x_i^* + \frac{1}{n}x_i^\Delta$.  Put all other indices to set $L'$, i.e. $L' = \{i \in \{1,\ldots,n\}\mid\mu[x] \in \bx_i^{1/n}\}$. Set $k = \arg\max_{k\in\{1,\ldots,m\}} \abs{\mu[x] - a_k}$. Consider the set $L$ processed by Algorithm \whole\ref{alg:traversal} in $k$th iteration of Algorithm \whole\ref{alg:fgkla}. By construction, $L' \subseteq L$. Hence, the maximizer $x$ is among the examined vertices.
        \item \emph{Complexity.} On Line \whole\ref{alg:sorting}, the algorithm sorts $2n$ numbers with complexity $O(n \log n)$. Algorithm \whole\ref{alg:traversal} is called at most $m$ times, where $m \le 2n = O(n)$. 
            Recall that $\omega$ is the size of the maximal clique of the FGKLA intersection graph. 
            In the $k$th iteration of the \textbf{for} cycle on Lines \whole\ref{alg:f:for:start} to \whole\ref{alg:f:for:end} of Algorithm \whole\ref{alg:fgkla} we have $\abs{L} = \abs{\{i\mid a_k \in \bx^{1/n}_i\}}$. Thus $\abs{L} \leq \omega$. 
            
            Algorithm \whole\ref{alg:traversal} performs exactly $2^{\abs{L}}$ iterations of the \textbf{while} cycle on Lines \whole\ref{alg:while:start} to \whole\ref{alg:while:end}. Inside its iteration, there is the \textbf{for} cycle on Lines \whole\ref{alg:for:start} to \whole\ref{alg:for:end}. The amortized time complexity of this \textbf{for} cycle is $O(1)$, because in its iteration it either sets some nonzero $z_i$ to $0$ or stops iterating. Since $z_i$ is set to a nonzero value only $2^{\abs{L}}$ times, the overall time of all courses of the \textbf{for} cycle is 
            $O(2^{\abs{L}})$. 

Computing time in the remaining steps is negligible. In particular, note that since $B_{a_1}, \ldots, B_{a_m}$ are pairwise disjoint sets (the same holds true for $E_{a_1},\ldots,E_{a_k}$), the total number of iterations of \textbf{for} cycles on Lines \whole\ref{alg:B:processing} to \ref{alg:E:processing} is at most $n$ during the whole course of FGKLA algorithm. 
            
The overall complexity is $O(n\log n + n \cdot 2^\omega)$.\qed
    \end{enumerate}
\end{proof}

\begin{remark}
    \label{rem:orig:impl:problem}
    Aside of the implementation details (which are important for the reduced time complexity bound), our formulation of the algorithm differs from the original paper \cite{ferson:2005:ExactBoundsFinitea} also for another reason. The original formulation can lead 
to complexity $O(n^2 \cdot 4^\omega)$, for example if $\omega = \ell$ and if there are $\ell$ narrowed intervals ending in some $a_k$ and further $\ell$ narrowed intervals starting in $a_{k+1}$. However, a minor modification of the original formulation would be sufficient to achieve the time $O(n^2 \cdot 2^\omega)$.
\end{remark}

\section{A probabilistic model}
\label{sect:probability}

This section is devoted to the main probabilistic result: on average, FGKLA algorithm works in ``almost'' linear time.

Here we use the statistical motivation of the problem as described in \whole\ref{sect:statistics}. Namely, in statistics, 
data are often assumed to form a random sample from a certain distribution. This is exactly our probabilistic model: 
we assume that both centers of the intervals and their radii form two independent random samples from fairly general classes of distributions.

\begin{assumption}[the probabilistic model]\label{ass:pro:model}\hfill
    \begin{enumerate}[label=(\Alph*),ref=\Alph*]
        \item \label{ass:A} The centers $x^*_1, \dots, x^*_n$ are independent and identically
distributed (``i.i.d.'') random variables with a Lipschitz continuous cumulative distribution function (``c.d.f.'')~$\Phi^*(z)$. 
That is, there exists a constant $L > 0$ such that 
$$\Phi^*(\widetilde z) - \Phi^*(z) \leq L(\widetilde z - z) \text{\ \ whenever\ \ } \widetilde z > z.
$$
\item\label{ass:B} The radii $x^\Delta_1, \dots, x^{\Delta}_n$ are i.i.d.~nonnegative random variables with a finite moment of order 
    $1+\varepsilon$
for some $\varepsilon > 0$. In other words, we assume
\begin{equation}
\gamma \coloneqq \mathsf{E}[(x_i^{\Delta})^{1+\varepsilon}] < \infty. \label{eq:moment}
\end{equation}
\item\label{ass:C} The random variables $x_i^{*}$, $x_i^{\Delta}$ are independent.
\end{enumerate}
\end{assumption}

\begin{theorem}\label{theo:main}
Denote by $\omega$ the size of the largest clique 
of the FGKLA intersection graph with data $[\ux_i \coloneqq x_i^* - x^{\Delta}_i,\ \ox_i \coloneqq x_i^* + x^{\Delta}_i]_{i = 1, \dots, n}$. 
If $n$ is sufficiently large, then
\begin{enumerate}[label=(\alph*)]
    \item\label{theo:main:a} $\EE 2^\omega \leq 1 + n^\frac{1}{\log \log n}$ and 
   $\EE \omega \leq \frac{3}{2}\left(1 + \frac{\log n}{\log \log n}\right)$,
     \item\label{theo:main:b} $\Pr[\omega \geq \delta n] \leq e^{-n \ln\ln n}$ for any $\delta > 0$.
\end{enumerate}
\end{theorem}

\begin{remark} Proof of \whole\ref{theo:main} will be given in \whole\ref{sec:proof:theo:main}. Statement (b) should be understood more precisely as follows: for every $\delta > 0$ there exists $n_{\delta}$ such that $\Pr[\omega \geq \delta n] \leq e^{-n \ln\ln n}$ if $n \geq n_\delta$.
\end{remark}

\begin{corollary}[main result] 
    The average computing time is 
\begin{align*}
O(\mathsf{E}[n\log n + n\cdot 2^{\omega}]) &= 
O(n\log n + n\cdot \mathsf{E}2^{\omega})  \\
&= O(n\log n + n\cdot n^{\frac{1}{\log\log n}}) \\  &= O(n^{1+\epsilon})
\end{align*} 
for an arbitrarily small $\epsilon > 0$.
Moreover, the computing time is $2^{\Omega(n)}$ when $\omega$ is linear in $n$ and this event 
occurs with probability as small as $O(e^{-n\ln\ln n})$.
\end{corollary}

\begin{remark}
    \whole\ref{ass:B} on the distribution of radii is very mild; indeed, we need just something a little more than
 existence of the expectation (we even do not need finite variance). On the other hand, Lipschitz continuity of~$\Phi^*$ (\whole\ref{ass:A})
is unavoidable; we will show what can happen without Lipschitz continuity in \whole\ref{sect:comments}. We will also discuss there
what happens when we relax the independence assumption (\whole\ref{ass:C}) and what is the cost for dependence paid by existence of higher-order moments.
\end{remark}

\subsection{Proof of \ref{theo:main}\label{sec:proof:theo:main}}
\paragraph{Notation.}
For a random variable $X$, its probability density function (``p.d.f.'') is denoted by $\varphi_X$.
Denote by $q_j$ the $j$th $n$-quantile of the distribution of centers: i.e., let $q_1, \dots, q_{n-1}$ satisfy
$$
\Phi^*(q_j) = \frac{j}{n}, \quad j = 1, \dots, n-1.
$$
Let $I_1, I_2, \ldots, I_{n-1}, I_n$ stand for the intervals $(-\infty,q_1],[q_1,q_2], \ldots,$ $[q_{n-2},q_{n-1}], [q_{n-1},\infty)$.
Consider the probabilities
$$
p^n_j \coloneqq \Pr[\bx_i^{1/n} \cap I_j \not = \emptyset];
$$
observe that $p^n_j$ does not depend on $i$ by the i.i.d.~assumptions.
Probabilities $p^n_1, \dots, p^n_n$ may differ due to the different shape of $\varphi_{x^*_i}$ around $q_1, \ldots, q_{n-1}$. 
In the upcoming lines, we utilize the fact
\begin{equation}
\varphi_{x^*_i}(z) \leq L\quad \forall z \label{eq:lips}
\end{equation}
implied by Lipschitz continuity of $\Phi^*(z)$ to derive the following upper bound:

\begin{lemma} For every $j=1,\ldots,n$ we have $p^n_j \leq \frac{\alpha}{n}$, where\label{lem:p}
\begin{equation}
\alpha \coloneqq 1+ 2 L\left(1+\frac{\gamma}{\varepsilon}\right).
\label{eq:alpha}
\end{equation}
\end{lemma}

\begin{remark}
The value of $\alpha$ depends on the properties of distributions of centers and radii. The ``hard'' cases are those with $\alpha \gg 1$. Indeed, the difficult case is when $\varepsilon$ is close to zero (``radii can be large with a high probability''), $\gamma \gg 0$ (``radii are large on average'') and $L \gg 0$ (``the density of centers can have high peaks'', or ``many centers can be close to one another'').

\end{remark}

\begin{proof}[of lemma~\ref{lem:p}]
    Let $n$ and $i$ be fixed; we omit the index $n$ for brevity. For $j=2,\ldots,n-1$, we can decompose $p^n_j \equiv p_j $ as
    \begin{align}
        p_j &= \Pr[\bx_i^{1/n} \cap I_j \not = \emptyset] \\
            &= \Pr[x^*_i \in I_j] + \underbrace{\Pr[x^*_i + \tfrac{1}{n}x^\Delta_i \ge q_{j-1} \ \land\  x^*_i < q_{j-1}]}_{\eqqcolon p^-_j} \nonumber \\
            &\phantom{= \Pr[x^*_i \in I_j]} \ + \underbrace{\Pr[x^*_i - \tfrac{1}{n}x^\Delta_i  \le q_j \ \land\ x^*_i > q_j]}_{\eqqcolon p^+_j}. \nonumber
    \end{align} 
    Observe that probability $p^-_j$ vanishes for $j=1$, as well as probability $p^+_j$ does for $j=n$.

    By definition of $I_j$, we have $\Pr[x^*_i \in I_j] = \frac{1}{n}$. 

    For $p^+_j$, we derive the upper bound 
    \ref{eq:p:plus}. Note that the bound does not depend on $j$. Note that the bound holds also true for $p^-_j$ by symmetry.

    By Markov's inequality and \ref{ass:B} we get
\begin{equation}
    \Pr[x_i^\Delta \geq z] \leq \frac{\gamma}{z^{1+\varepsilon}}. \label{eq:markoff}
\end{equation}

Setting $x_j' := n(x^*_i - q_j)$, from \ref{eq:lips} we have 
\begin{equation}
\varphi_{x'_j}(z) \le \frac{L}{n}\quad \forall z.
\label{eq:lipboundx}
\end{equation}
Now
    \begin{equation}
        \label{eq:p:plus}
        \begin{aligned}
            p^+_j &= \Pr[x^*_i - \tfrac{1}{n}x^\Delta_i \le q_j \ \land\ x^*_i > q_j]\\
                &= \Pr[x_i^\Delta \ge x'_j \ \land\ x'_j > 0] \\
                &= \int_0^\infty \Pr[x_i^\Delta \ge z\ |\ x'_j = z] \cdot \varphi_{x'_j}(z)\, \mathrm{d}z\\
                &= \int_0^\infty \Pr[x_i^\Delta \ge z] \cdot \varphi_{x'_j}(z) \, \mathrm{d}z\\
                &= \int_0^1 \Pr[x_i^\Delta \ge z] \cdot \varphi_{x'_j}(z)\, \mathrm{d}z+ \int_1^\infty \Pr[x_i^\Delta \ge z] \cdot \varphi_{x'_j}(z)\, \mathrm{d}z\\
                &\le \int_0^1 \varphi_{x'_j}(z)\, \mathrm{d}z+ \int_1^\infty \frac{\gamma}{z^{1+\varepsilon}} \cdot \varphi_{x'_j}(z)\, \mathrm{d}z\\
                &\le \frac{L}{n} + \frac{L\gamma}{n} \int_1^\infty \frac{1}{z^{1+\varepsilon}}\, \mathrm{d}z\\
                & = \frac{L}{n} + \frac{L\gamma}{\varepsilon n} \\
                & = \frac{L}{n}\left(1+\frac{\gamma}{\varepsilon}\right).
    \end{aligned}
    \end{equation}

Finally,
$$p_{j} \le \frac{1}{n} + 2 \frac{L}{n}\left(1+\frac{\gamma}{\varepsilon}\right) = \frac{1}{n} \left(1+ 2 L\left(1+\frac{\gamma}{\varepsilon}\right)\!\right) = \frac{\alpha}{n}.\qed
$$
\end{proof}

Let us introduce indicator variables for all $i,j = 1, \dots, n$:
$$
W_{ij} = \left\{ 
\begin{array}{ll}
    1 & \text{if $\bx^{1/n}_i \cap I_j \not = \emptyset$}, \\
0 & \text{otherwise.}
\end{array}\right.
$$
Note that $W_{ij}$ is alternatively distributed with parameter $p^n_j$
and that the variables 
\begin{equation}
W_{1j}, W_{2j}, \dots, W_{n-1,j}, W_{nj}\label{eq:wecka}
\end{equation}  
are independent.

Hence, the sum of these variables 
follows binomial distribution, i.e.
$$\sum_{i=1}^n W_{ij} \eqqcolon E_j \sim \textrm{Bi}(n,p^n_j).$$ By introducing $$\overline{E} \sim \textrm{Bi}(n,\tfrac{\alpha}{n}),$$ the estimate $p^n_j \leq \tfrac{\alpha}{n}$ from~\ref{lem:p} implies
\begin{equation}\Pr[E_j \geq z] \leq \Pr[\overline{E} \geq z] \quad (\forall z).\label{eq:bin}\end{equation}

If $\omega \ge \kappa$ for some $\kappa$, at least $\kappa$ intervals $\bx^{1/n}_i$ have to share a common intersection. This common intersection belongs to one of intervals $I_1,\ldots, I_n$, hence there exists $j$ such that $E_j \ge \kappa$. More precisely,
\begin{equation}
    \Pr[\omega \ge \kappa] \le \Pr[E_1 \ge \kappa \lor \dots \lor E_n \ge \kappa].\label{eq:penrose}
\end{equation}

Fact \ref{eq:bin} allows us to use an estimate based on Penrose's tail bound for binomial distribution. 

\begin{lemma}[Tail bound for the binomial distribution {(\cite[p. 16]{penrose:2003:RandomGeometricGraphs}})]\label{lem:tail}
Let 
\begin{equation}
H(\xi) = 1 - \xi + \xi\log \xi.\label{eq:H}
\end{equation}
If $Z \sim \mathrm{Bi}(n,\pi)$ and $\kappa \geq n\pi$, then 
$$
\Pr[Z \geq \kappa] \leq \exp\left(-n\pi \cdot H\left(\frac{\kappa}{n\pi}\right)\right).
$$
\end{lemma}

If \begin{math}
\label{eq:assproodhad}
\kappa \geq \alpha,
\end{math} 
we can extend the estimate \whole\ref{eq:penrose} to the form

\begin{align}
\Pr[\omega \geq \kappa] & \leq \Pr[E_1 \geq \kappa \vee \cdots \vee E_n \geq \kappa] \label{eq:first} 
                     \\ & \leq \sum_{j=1}^n \Pr[E_j \geq \kappa] \label{eq:second} 
                     \\ & \leq \sum_{j=1}^n \Pr[\overline{E} \geq \kappa]    \label{eq:bndxx}                  
                     \\ & \leq n\exp\left[-n \frac{\alpha}{n}\cdot H\left(\frac{\kappa}{n \frac{\alpha}{n}}\right)\right] \nonumber
\\ & = n\exp\left[-\alpha\cdot H\left(\frac{\kappa}{\alpha}\right)\right]. \label{eq:mclast}
\end{align}

In \ref{eq:second} we used 
the union bound 
$\Pr[Q_1 \vee \cdots \vee Q_n] \leq \sum_{i=1}^n \Pr[Q_i]$ for any events $Q_1, \dots, Q_n$.
In \ref{eq:bndxx} we used \ref{eq:bin} and the tail bound from~\ref{lem:tail}.

Let
\begin{equation}\label{eq:kn:definition}
c \coloneqq \frac{1}{\log 2} \quad \text{and}\quad k_n \coloneqq c \cdot\frac{\log n}{\log \log n}.
\end{equation}

\newcommand\numberthis{\addtocounter{equation}{1}\tag{\theequation}}
If $n$ is sufficiently large, we can estimate
\newcommand{\pil}{\Pr[\omega=\ell]}

\begingroup
\allowdisplaybreaks
\begin{align*}
    \EE2^\omega &= \sum_{\ell = 1}^n 2^\ell \cdot \pil\\
    &= \sum_{\ell = 1}^{\lfloor k_n \rfloor} 2^\ell \cdot \pil&&  + \sum_{\ell = \lfloor k_n \rfloor + 1}^{n} 2^\ell \cdot \pil\\ \\ 
    &\le \sum_{\ell=1}^{\lfloor k_n \rfloor} 2^{\lfloor k_n \rfloor} \cdot \pil&& +\sum_{\ell = \lfloor k_n \rfloor + 1}^{n} 2^\ell \cdot \Pr[\omega\ge\ell] \\
    &\leq 2^{k_n} \sum_{\ell=1}^{\lfloor k_n \rfloor} \pil   \\    
\span \span\span +\sum_{\ell = \lfloor k_n\rfloor+1}^{n} \underbrace{2^{\ell} \cdot n\cdot \exp\left(-\alpha\cdot H\left(\frac{\ell}{\alpha}\right)\right)}_{\eqqcolon u_\ell}\\
                &\leq 2^{k_n}
                &&+\sum_{\ell = \lfloor k_n\rfloor+1}^{n} u_\ell\\ 
                &\le 2^{\frac{c\log n}{\log \log n}}                  
            &&+2 u_{\lfloor k_n \rfloor + 1} \numberthis\label{eq:boundpred} \\
            &\le e^{(\log 2)\cdot \frac{c\log n}{\log \log n}}  
            &&+ 4 \cdot \underbrace{n^{1-c}n^{c\frac{K+\log \log\log n}{\log \log n}}}_{\eqqcolon \zeta_n}\numberthis\label{eq:bound}\\
            &\le n^\frac{1}{\log \log n}  && + 1, \numberthis\label{eq:boundtwo} \\
\end{align*}
\endgroup

where $K \coloneqq \log\frac{8\alpha}{c}$. 
Clearly, $1 + n^\frac{1}{\log \log n} =  O(n^\epsilon)$ for every $\epsilon >0$.
Inequalities~\whole\ref{eq:boundpred} and \whole\ref{eq:bound} follow from \whole\ref{lem:u:i} showing basic properties of the sequence~$u_{\ell}$. Namely, it shows that it decreases exponentially fast.
The estimate $\zeta_n\leq \frac{1}{4}$ if $n$ is sufficiently large $n$ follows from the observation that $1 - c < 0$; thus
$\zeta_n \stackrel{n \rightarrow \infty}{\longrightarrow} 0$.

\newpage
\begin{lemma}\label{lem:u:i} Let $k_n \geq 4e\alpha$ (here, $e = \exp(1)$).
    \begin{enumerate}[label=(\alph*)]
        \item  $\sum_{\ell = \lfloor k_n \rfloor + 1}^{n} u_{\ell} < 2u_{\lfloor k_n \rfloor + 1}$.
\item $u_{\lfloor k_n \rfloor + 1} \leq 2\zeta_n$.
\end{enumerate}
\end{lemma}
\begin{proof}
    To prove (a) we show that $u_\ell<\frac{1}{2}u_{\ell-1}$ for 
    \begin{equation} 
    \ell \ge 4e\alpha\label{eq:fourea}.
    \end{equation}
    It follows that $$\sum_{\ell = \lfloor k_n \rfloor + 1}^{n} u_{\ell} \le 2u_{\lfloor k_n\rfloor +1},$$
    since $\sum_{\ell=\lfloor k_n\rfloor+1}^{n}u_\ell$ can be bounded by the sum of a geometric sequence with quotient~$\frac{1}{2}$. 
We have
\begingroup
\allowdisplaybreaks
\begin{align*}
u_\ell &=n\cdot2^{\ell} \cdot \exp\Big(-\alpha\cdot H\Big(\frac{\ell}{\alpha}\Big)\Big)\\
    &= n\cdot2^{\ell} \cdot \exp\Big[-\alpha\Big(1-\frac{\ell}{\alpha} 
    +\frac{\ell}{\alpha}\log\Big(\frac{\ell}{\alpha}\Big)\Big)\Big] \\
    &= n\cdot2^{\ell} \cdot \exp\left[-\alpha+\ell-\ell\log\ell+\ell\log \alpha\right]  \\
    &= n e^{-\alpha}\cdot\left(\frac{2e\alpha}{\ell}\right)^{\ell}
\end{align*}
\endgroup
and
$$
\frac{u_\ell}{u_{\ell-1}} = \left(\frac{\ell-1}{\ell}\right)^{\ell-1} \cdot \frac{2e\alpha}{\ell}<
\frac{2e\alpha}{\ell}\le\frac{1}{2};
$$
the last inequality follows from~\ref{eq:fourea}.
For (b) we use the fact that $H(\xi) = 1 -\xi + \xi\log\xi$ is a nondecreasing function for $\xi \geq 1$. Thus
\begingroup
\allowdisplaybreaks
\begin{align*}
    u_{\lfloor k_n \rfloor + 1} 
&=2^{\lfloor k_n \rfloor + 1} \cdot n\exp\left[-\alpha\cdot H\left(\frac{\lfloor k_n \rfloor + 1}{\alpha}\right)\right] 
\\ & \leq
2\cdot 2^{k_n} \cdot n\exp\left[-\alpha\cdot H\left(\frac{k_n}{\alpha}\right)\right] 
\\ & \leq
2\cdot e^{k_n} \cdot n\exp\left[-\alpha\cdot H\left(\frac{k_n}{\alpha}\right)\right] 
\\ & =
2\exp\Bigg[\log n + k_n 
 -\alpha\cdot \left(1 - \frac{k_n}{\alpha} + \frac{k_n}{\alpha} \log \left( \frac{k_n}{\alpha} \right)   \right)\Bigg] 
\\ & =
2\exp\left[ \log n + k_n -\alpha + k_n - k_n \log k_n + k_n\log \alpha\right]
  \\& \leq
2\exp\left[ \log n + (2 + \log\alpha)k_n - k_n \log k_n\right]
\\ & \leq
2\exp\left[ \log n + (\log8\alpha)k_n - k_n \log k_n\right]
\\ & =
2\exp\Bigg[ \log n + (\log8\alpha)\frac{c\log n}{\log\log n}
- \frac{c\log n}{\log\log n} \log \frac{c\log n}{\log\log n}\Bigg]
\\\phantom{u_{\lfloor k_n \rfloor +1}}& =
2\exp\Bigg[ \log n + (\log8\alpha)\frac{c\log n}{\log\log n}
\\&\qquad\qquad
- \frac{c\log n}{\log\log n} \left(\log c + \log\log n - \log\log\log n\right)\Bigg]
\\
& =
2\exp\Bigg[ \log n + (\log8\alpha)\frac{c\log n}{\log\log n}
\\&\qquad\qquad
- \frac{(c\log c)\log n}{\log\log n} - c\log n + (c \log n) \frac{\log\log\log n}{\log\log n}\Bigg]
\\ & =
2\exp\Big[(\log n)\Big((1-c) \\ & \qquad\qquad + c \cdot \frac{\log 8\alpha - \log c + \log\log\log n}{\log\log n}\Big)\Big]
\\& =
2\exp\left[(\log n)\left((1-c) + c \cdot \frac{K + \log\log\log n}{\log\log n}\right)\right]
\\ & =
2n^{1-c} \cdot n^{c \cdot \frac{K + \log\log\log n}{\log\log n}}
= 2\zeta_n.{\qed}
\end{align*}
\endgroup
\end{proof}

\begin{remark}\label{rem:choice:c}

Note that a huge $n$ might be needed to achieve $\zeta_n \le \frac{1}{4}$; the particular value depends on $\alpha$. 
However, if we admit a greater~$c$, e.g. $c = 8\alpha$, we get $K=0$ and $\zeta_n \leq n^{1 - 8\alpha + 8\frac{\alpha}{e}} \approx n^{1 - 5\alpha}$, which tends to zero fast, so condition $\zeta_n \le \frac{1}{4}$ is not at all restrictive even from the practical viewpoint. On other hand, the exponent in $n^{\frac{c \log 2}{\log\log n}} + 1$ becomes a bit worse. 

This shows that at the cost of a pair of worse constants, the method behaves well even for small $n$.

\end{remark}
\newpage
In order to complete the proof of \whole\ref{theo:main}(a), we need to estimate $\mathsf{E}\omega$. Using Jensen's inequality we get
\begin{align*}
\mathsf{E}\omega &\leq \frac{\log(\mathsf{E}2^{\omega})}{\log 2} \\
&\leq \frac{\log(1 + n^{\frac{1}{\log\log n}})}{\log 2} \\ 
&\leq \frac{3}{2}\left(1+ \log e^{\frac{\log n}{\log\log n}}\!\right)\! \\
&=\frac{3}{2}\Big(1 + \frac{\log n}{\log\log n}\Big)\!.
\end{align*}

The proof of \whole\ref{theo:main}(b) 
is a corollary of the above theory. Indeed,
using the notation from \whole\ref{eq:first} -- \whole\ref{eq:mclast}, definition of $H(\xi)$ from \whole\ref{eq:H}  and \whole\ref{lem:tail}, we have
\begingroup
\allowdisplaybreaks
\begin{align*}
& \frac{\Pr[\omega \geq \delta n]}{e^{-n\log\log n}} \leq \sum_{j=1}^n \Pr[\overline{E} \geq \delta n] \cdot e^{n\log\log n}
\\& \leq n \exp\left[-\alpha \cdot H\left( \frac{\delta n}{\alpha} \right)\right] \cdot e^{n \log\log n}
\\& = \exp\Bigg[\log n + n\log\log n 
-\alpha  \cdot \left(1 - \frac{\delta n}{\alpha} + 
          \frac{\delta n}{\alpha} \log \left( \frac{\delta n}{\alpha} \right)\right)\Bigg] 
\\ & \leq \exp\Bigg[\log n + n\log\log n + \delta n - \underbrace{\delta n \log \left(\frac{\delta}{\alpha}n\right)}_{(\star)} \Bigg] 
 \stackrel{n\rightarrow \infty}{\longrightarrow} 0,          
\end{align*}
\endgroup
because the term $(\star)$ is of the order $n\log n$ and dominates all other terms in the limit. 
The proof of \whole\ref{theo:main} is complete. \hfill$\qedsymbol$

\begin{remark} The same proof method can be easily generalized to estimate, for example, the probability that
the clique is as large as $n^{\eta}$ for a fixed $0 \leq \eta \leq 1$ 
(i.e., this is the event ``the computing time exceeds $2^{n^{\eta}}$''). In this case
we get $\Pr[\omega \geq n^{\eta}] \leq \exp(-n^{\eta}\log\log n)$.
\end{remark}

\section{Concluding remarks and comments}\label{sect:comments}

\subsection{``Unfriendly'' distributions for the FGKLA algorithm: Why Lipschitz continuity (\ref{ass:A}) is unavoidable} 

We show that if we drop the Lipschitz continuity assumption,
we can get $\mathsf{E}\omega \geq \pi n$ for some $\pi > 0$ 
and thus exponential computing time on average 
(using the fact that \emph{average computing time} $ \geq \mathsf{E}2^{\omega} \geq 2^{\mathsf{E}\omega} = 2^{\Omega(n)}$). 

\paragraph{Non-continuous distributions.}
First consider $\Phi^*(z)$, the c.d.f. of $x_i^*$, with a discontinuity point $z_0$.
Then $\pi \coloneqq \Pr[x_i^* = z_0] > 0$. Setting
\begin{equation}
U_i = \left\{\begin{array}{ll}
1, & \text{if $z_0 \in \bx_i^{1/n}$,} \\
0  & \text{otherwise},
\end{array}\right.\label{eq:Ui}
\end{equation}
we get $\mathsf{E}U_i = \Pr[U_i = 1] \geq \pi$ and $\omega \geq \sum_{i=1}^n U_i$ a.s. Thus
$$
\mathsf{E}\omega \geq 
\sum_{i=1}^n \mathsf{E}U_i \geq \pi n.
$$

\paragraph{Continuous non-Lipschitz distributions.} We show that the misbehavior of
the non-continuous distribution from the previous paragraph can be ``simulated'' by a non-Lipschitz continuous distribution. Let $z_0$ be a discontinuity point of $\Phi^*(z)$ from the last paragraph,
let $\Phi_0 \coloneqq \lim_{z \nearrow z_0} \Phi^*(z)$ and
$\eta \coloneqq \Phi^*(z_0 + 1) - \Phi_0$. Clearly $\eta > 0$.
Consider another distribution of $x_i^*$ with c.d.f.
$$
\widetilde\Phi^*(z) = \left\{
\begin{array}{ll}
\Phi^*(z) & \text{if $z < z_0$ or $z > z_0 + 1$}, \\
\Phi_0 + \eta \cdot (z-z_0)^{\varepsilon} & \text{if $z_0 \leq z \leq z_0 + 1$}
\end{array}
\right.
$$
with $\varepsilon > 0$ arbitrarily small. Now $\widetilde\Phi^*(z)$ is continuous (if there are more discontinuity points of $\Phi^*(z)$ outside $[z_0, z_0+1)$, a similar construction can be done in each of them).
If $x_i^{\Delta} = 1$ a.s. and
$U_i$ has the same meaning as in \whole\ref{eq:Ui}, we get
$$
\mathsf{E}U_i = \Pr[U_i = 1]
\geq 
\Pr[z_0 \leq x_i^* \leq z_0 + \tfrac{1}{n}]
= \eta n^{-\varepsilon},
$$
and thus
$$
\mathsf{E}\omega \geq 
\sum_{i=1}^n \mathsf{E}U_i \geq \eta n^{1 - \varepsilon}.
$$
Taking $\varepsilon$ close to zero, we get a clique with average 
size arbitrarily close to the order $n$.

\subsection{The independence assumption (\ref{ass:C}) is also essential} \label{ssect:C}
If we relax the independence assumption, we can get only a weaker estimate on $p^n_j$ than the bound 
$p^n_j = O(n^{-1})$ from \whole\ref{lem:p}. 
Said informally, we needed $p^n_j = O(n^{-1})$ in 
\whole\ref{lem:tail} to satisfy $np^n_j = O(1)$. Then, since $k_n$ grows unboundedly (although slowly), 
we were able to apply the tail bound for
$n$ sufficiently large. 

But in the dependent case we can derive only the bound 
\begin{equation}
    p^n_j = O(n^{-\frac{1}{2}}), \label{eq:bnd}
\end{equation}
resulting in $n p^n_j = O(n^{\frac{1}{2}})$. Then, $k_n$ would have to
grow faster than $n^{\frac{1}{2}}$ to be able to apply the tail bound and we would get 
\begin{equation}
    \text{$\mathsf{E}\omega \sim n^{\frac{1}{2}}$}\label{eq:bigom}
\end{equation}
or even something worse. 
Then, the average computation time bound would be as poor as $2^{\sqrt{n}}$. This is a high price for dependence. For specific extremal distributions, the situation can indeed 
be so bad, as shown in \whole\ref{sect:extr};
but for ``usual'' distributions with enough moments the situation is much better, 
as explained in \whole\ref{sect:moments}.

Let us show \whole\ref{eq:bnd} without the assumption of independence of $x^*_i$ and $x_i^\Delta$. By Markov's inequality we have 
$$
\zeta_n \coloneqq \Pr[x_i^\Delta \geq n^{\frac{1}{2+\varepsilon}}] 
\leq \gamma n^{-\frac{1+\varepsilon}{2+\varepsilon}}$$ 
similarly as in \whole\ref{eq:markoff}; recall that we have only assumed 
the existence of a finite moment of order $(1+\varepsilon)$ with value $\gamma$. 
We have
\begin{align*}
    p^n_j &= \Pr[x^*_i \in I_j] + p^+_j + p^-_j 
         \\&\le \tfrac{1}{n} + O(n^{-\frac{1}{2}}) + O(n^{-\frac{1}{2}}) 
         \\&= O(n^{-\frac{1}{2}}),
\end{align*}
as the bounds for $p^+_j$ (and $p^-_j$, similarly) can be obtained by setting $x'_j = n(x_i^* - q_j)$
and using \ref{eq:lipboundx}:
\begin{align*}
    p^+_j &= \Pr[x'_j \le x_i^\Delta\ \land\ x'_j \ge 0] 
\\     &= \Pr[x'_j \le x_i^\Delta\ \land\ x'_j \ge 0\ |\ x_i^\Delta \ge n^{\frac{1}{2+\varepsilon}}] \cdot \zeta_n
      \\&\qquad + \Pr[x'_j \le x_i^\Delta\ \land\ x'_j \ge 0\ |\ x_i^\Delta < n^{\frac{1}{2+\varepsilon}}] \cdot (1-\zeta_n)
     \\ & \le\zeta_n + \Pr[x'_j < n^{\frac{1}{2+\varepsilon}} \land x'_j \ge 0]
     \\ & \le \gamma n^{-\frac{1+\varepsilon}{2+\varepsilon}} + L n^{-1}n^{\frac{1}{2+\varepsilon}}
     \\ & = \gamma n^{-\frac{1+\varepsilon}{2+\varepsilon}} + L n^{-\frac{1+\varepsilon}{2+\varepsilon}}
     \\ & = O(n^{-\frac{1}{2}}).
\end{align*}

\subsection{An extremal distribution}\label{sect:extr} 
\hspace{-0.2cm} Unfortunately, the bounds from the previous sections cannot be generally improved. We show an example where 
\whole\ref{ass:A} and \ref{ass:B} are satisfied,
\whole\ref{ass:C} is violated 
and
a slightly weaker form of \whole\ref{eq:bigom} holds true ---
the clique is as large as $n^{\frac{1}{2} - \varepsilon}$ on average, for an arbitrarily small $\varepsilon > 0$. 
Thus we can push the average computation time of FKGLA algorithm arbitrarily close
to $2^{\sqrt n}$.

Let $x_1^*, \dots, x_n^* \sim \text{Unif}(0,1)$ independent. Then, clearly, \whole\ref{ass:A} is satisfied.
Let $0 < \varepsilon < 1$ (a choice with $\varepsilon$ close to zero is interesting).
Define
$$
x_i^{\Delta} \coloneqq (x_i^*)^{-1+\varepsilon}, \quad i = 1, \dots, n.
$$
\whole\ref{ass:B} is satisfied: indeed, the moment of order $1+\varepsilon$ is finite, since
\begin{align*}
\mathsf{E}[(x_i^{\Delta})^{1+\varepsilon}] 
& = 
\mathsf{E}[(x_i^*)^{(\varepsilon -1)(\varepsilon + 1)}]
\\ &=
\mathsf{E}[(x_i^*)^{\varepsilon^2 - 1}] 
\\ &  =
\int_0^1 x^{\varepsilon^2 - 1}\ \text{d}x \\ &= \varepsilon^{-2} < \infty,
\end{align*}
and $x_1^{\Delta}, \dots, x_n^{\Delta}$ are independent.

For $i = 1, \dots, n$ define
$$
U_i = \left\{
\begin{array}{ll}
1, & \text{if $0 \in \bx_i^{1/n}$,} \\
0  & \text{otherwise.} 
\end{array}
\right.
$$
We have 
\begin{align*}
\mathsf{E}U_i 
&= 
\Pr[U_i = 1]
=
\Pr[\tfrac{1}{n}x_i^{\Delta} \geq x_i^*]
=
\Pr[(x_i^*)^{-1+\varepsilon} \geq nx_i^*]
\\ & =
\Pr[(x_i^*)^{-2+\varepsilon} \geq  n]
 =
\Pr[x_i^* \leq n^{-\frac{1}{2-\varepsilon}}]
=
n^{-\frac{1}{2-\varepsilon}}.
\end{align*}
Obviously, $\omega \geq \sum_{i=1}^n U_i$ a.s. Thus
$$
\mathsf{E}\omega 
\geq
\sum_{i=1}^n \mathsf{E}U_i
= n^{1 - \frac{1}{2-\varepsilon}},
$$
which is close to $n^{\frac{1}{2}}$ if $\varepsilon$ is small.

\subsection{FGKLA algorithm can benefit from high-order moments: A trade-off between dependence and existence of such moments}\label{sect:moments}
Note that the problem with dependence of the input random variables sketched in \whole\ref{ssect:C} is closely related to the value of $\varepsilon$. Here we show that

if we assume the existence of high-order moments, we can push the bound on $p^n_j$ close to the ``desired'' order $O(n^{-1})$ and get good computation time of the FGKLA algorithm even in the dependent case. 
Indeed, if $\widetilde\gamma \coloneqq \mathsf{E}[(x_i^{\Delta})^d] < \infty$ for some $d$, then
Markov's inequality gives us $\zeta_n \coloneqq \Pr[x_i^{\Delta} \geq n^{\frac{1}{1+d}}] \leq \widetilde\gamma n^{-\frac{d}{1+d}}$. Now we have 
\begin{align*}
    p^+_j &= \Pr[x'_j \le x_i^\Delta\ \land\ x'_j \ge 0] 
\\     &= \Pr[x'_j \le x_i^\Delta\ \land\ x'_j \ge 0\ |\ x_i^\Delta \ge n^{\frac{1}{1+d}}] \cdot \zeta_n
      \\&\qquad + \Pr[x'_j \le x_i^\Delta\ \land\ x'_j \ge 0\ |\ x_i^\Delta < n^{\frac{1}{1+d}}] \cdot (1-\zeta_n)
     \\ & \le\zeta_n + \Pr[x'_j < n^{\frac{1}{1+d}}\ \land\ x'_j \ge 0]
     \\ & \le \widetilde\gamma n^{-\frac{d}{1+d}} + L n^{-1}n^{\frac{1}{1+d}}
     \\ & \le \widetilde\gamma n^{-\frac{d}{1+d}} + L n^{-\frac{d}{1+d}}
     \\ & = O(n^{-\frac{d}{1+d}}),
\end{align*}

which is close to $n^{-1}$ if $d$ is large.

\begin{acknowledge}

All authors were supported by the Czech Science Foundation pro-ject 22-19353S. O.~Sokol was also supported by the Internal Grant Agency (project F4/19/2019) of Faculty of Informatics and Statistics, Prague University of Economics and Business. 
\end{acknowledge}

\end{document}